\documentclass{article}
\usepackage{amsmath,amsthm,amsfonts}
\usepackage{url}
\usepackage[T2A,T1]{fontenc}
\usepackage[utf8]{inputenc}
\usepackage[russian,english]{babel}
\usepackage{babelbib}
\usepackage{dsfont}
\usepackage{tikz}
\usepackage{hyperref}

\theoremstyle{theorem}
\newtheorem{theorem}{Theorem}
\newtheorem{lemma}[theorem]{Lemma}
\newtheorem{corollary}[theorem]{Corollary}

\theoremstyle{definition}

\newtheorem{exercise}{Exercise}

\newcommand{\RR}{\mathbb{R}}
\DeclareMathOperator{\sgn}{sgn}
\DeclareMathOperator{\rank}{rank}
\DeclareMathOperator{\dih}{dih}
\DeclareMathOperator{\sig}{sig}

    \makeatletter
    \let\@fnsymbol\@arabic
    \makeatother

\begin{document}
\title{Sarrus rules and dihedral groups}
\author{Dirk A. Lorenz\thanks{d.lorenz@tu-braunschweig.de} \quad Karl-Joachim Wirths\thanks{k-j.wirths@tu-braunschweig.de}\\ Institute of Analysis and Algebra\\
TU Braunschweig}
\maketitle

% \begin{abstract}
% This paper is devoted to the analysis of a false generalization of the rule of Sarrus and its properties that can be derived with the help of dihedral groups. Further, we discuss a Sarrus-like scheme that could be helpful for students to memorize the calculation of a $4\times 4$ determinant.
% \end{abstract}

\noindent
Beginners in Linear Algebra learn to calculate the $3\times 3$ determinant by the rule of Sarrus. 
For a matrix $A = (a_{ij})$, the determinant is
\[
\det(A)\! = \!a_{11}a_{22}a_{33} + a_{12}a_{23}a_{31} + a_{13}a_{21}a_{32} - a_{13}a_{22}a_{31} - a_{11}a_{23}a_{32} - a_{12}a_{21}a_{33}
\]
and the six summands are found in the following scheme %which is particularly easy to memorize:
\begin{center}
  \includegraphics{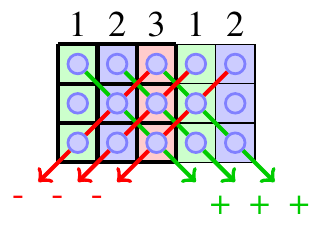}
  % \begin{tikzpicture}[xscale=0.4,yscale=-0.4,
  %   place/.style={circle,inner sep=2pt,draw=blue!50,fill=blue!20,thick},
  %   positive/.style={very thick,green!80!black,->},
  %   negative/.style={very thick,red,->}]

  %   \fill[green!20] (0,0) rectangle (1,3);

  %   \fill[blue!20] (1,0) rectangle (2,3);

  %   \fill[red!20] (2,0) rectangle (3,3);

  %   \fill[green!20] (3,0) rectangle (4,3);

  %   \fill[blue!20] (4,0) rectangle (5,3);

  %   \draw[very thick] (0,0) grid (3,3);

  %   \draw (3,0) grid (5,3);

  %   \foreach \i in {1,2,3}{ 
  %     \foreach \j in {1,2,3,4,5} \node (a\i\j)
  %     at (\j-1+0.5,\i-1+0.5) [place] {};
  %   };
    
  %   \node at (0.5,-0.5){1};
  %   \node at (1.5,-0.5){2};
  %   \node at (2.5,-0.5){3};
  %   \node at (3.5,-0.5){1};
  %   \node at (4.5,-0.5){2};
    
  %   \draw[positive] (a11) -- (a22) -- (a33) -- (3.5,3.5) node [below right] {+};
    
  %   \draw[positive] (a12) -- (a23) -- (a34) -- (4.5,3.5) node [below right] {+};
    
  %   \draw[positive] (a13) -- (a24) -- (a35) -- (5.5,3.5) node [below right] {+};
        
  %   \draw[negative] (a13) -- (a22) -- (a31) -- (-0.5,3.5) node [below left] {-};
    
  %   \draw[negative] (a14) -- (a23) -- (a32) -- (0.5,3.5) node [below left] {-};
    
  %   \draw[negative] (a15) -- (a24) -- (a33) -- (1.5,3.5) node [below left] {-};
  % \end{tikzpicture}
\end{center}
known as ``Rule of Sarrus''.
The right generalization of this rule to larger square matrices is
Leibniz' explicit formula for the determinant which is quite
complicated as it contains, for the $n\times n$ case, $n!$ summands where each summand is a
product of $n$ entries of the matrix (see below for a formula). Hence,
there seems to be a natural tendency of students to use the following
generalization of Sarrus' rule: to
calculate the determinant of a $n\times n$ matrix, the $n-1$ first
columns are repeated behind the matrix. Then the sum of the product of
the $n$ diagonals of this scheme from the upper left to the lower
right is calculated, and the sum of the $n$ diagonals from the upper right
to the lower left is taken and the latter sum is subtracted from
the former sum as here:
\begin{center}
  \includegraphics{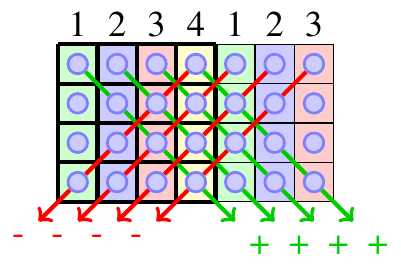}
  % \begin{tikzpicture}[xscale=0.4,yscale=-0.4,
  %   place/.style={circle,inner sep=2pt,draw=blue!50,fill=blue!20,thick},
  %   positive/.style={very thick,green!80!black,->},
  %   negative/.style={very thick,red,->}]
  %   \fill[green!20] (0,0) rectangle (1,4);
  %   \fill[blue!20] (1,0) rectangle (2,4);
  %   \fill[red!20] (2,0) rectangle (3,4);
  %   \fill[yellow!20] (3,0) rectangle (4,4);
  %   \fill[green!20] (4,0) rectangle (5,4);
  %   \fill[blue!20] (5,0) rectangle (6,4);
  %   \fill[red!20] (6,0) rectangle (7,4);
  %   \draw[very thick] (0,0) grid (4,4);
  %   \draw (4,0) grid (7,4);
  %   \foreach \i in {1,2,3,4}{
  %     \foreach \j in {1,2,3,4,5,6,7}
  %     \node (a\i\j) at (\j-1+0.5,\i-1+0.5) [place] {};
  %   };

  %   \node at (0.5,-0.5){1};
  %   \node at (1.5,-0.5){2};
  %   \node at (2.5,-0.5){3};
  %   \node at (3.5,-0.5){4};
  %   \node at (4.5,-0.5){1};
  %   \node at (5.5,-0.5){2};
  %   \node at (6.5,-0.5){3};
    
  %   \draw[positive] (a11) -- (a22) -- (a33) -- (a44) -- (4.5,4.5) node[below right]{+};
    
  %   \draw[positive] (a12) -- (a23) -- (a34) -- (a45) -- (5.5,4.5) node[below right]{+};
    
  %   \draw[positive] (a13) -- (a24) -- (a35) -- (a46) -- (6.5,4.5) node[below right]{+};
   
  %   \draw[positive] (a14) -- (a25) -- (a36) -- (a47) -- (7.5,4.5) node[below right]{+};
        
  %   \draw[negative] (a14) -- (a23) -- (a32) -- (a41) -- (-0.5,4.5) node[below left]{-}; 
    
  %   \draw[negative] (a15) -- (a24) -- (a33) -- (a42) -- (0.5,4.5) node[below left]{-};
    
  %   \draw[negative] (a16) -- (a25) -- (a34) -- (a43) -- (1.5,4.5) node[below left]{-};
    
  %   \draw[negative] (a17) -- (a26) -- (a35) -- (a44) -- (2.5,4.5) node[below left]{-};
  % \end{tikzpicture}
\end{center}
This amounts to the following formula:
\[
\begin{split}
  a_{11}a_{22}a_{33}a_{44} + a_{12}a_{23}a_{34}a_{41} +
  a_{13}a_{24}a_{31}a_{42} + a_{14}a_{21}a_{32}a_{43}\\
  \qquad  -  a_{14}a_{23}a_{32}a_{41} - a_{11}a_{24}a_{33}a_{42} - a_{12}a_{21}a_{34}a_{43} - a_{13}a_{22}a_{31}a_{44}.
\end{split}
\]
In this article we will give a generalization of this
formula to the $n\times n$ case and we will call this formula the \emph{False Sarrus Rule}. 
The quantity that it calculates will be called \emph{dihedrant} for reasons that will become clear below. Then, we will derive some
properties of the determinant that are similar to certain properties of the dihedrant. Especially, we are interested in
conditions that guarantee that this quantity will be equal to
$\det(A)$. Further, we shall show a scheme, which resembles Sarrus'
rule and can help to calculate the correct determinant of a
$4\times 4$ matrix.

\section{The determinant and the ``dihedrant''.}

Let $A\,=\,(a_{i,j}), i,j\,=\,1,\cdots, n$. The \emph{Leibniz formula} for the determinant is the following sum which runs over all permutations of the indices $\{1,\dots,n\}$, i.e., over the whole symmetric group $S_{n}$:
\begin{equation*}
  \label{eq:leibniz}
  %\tag{Leibniz}
  \det(A) = \sum_{\sigma\in S_{n}}\sgn(\sigma)\prod_{i=1}^{n}a_{i,\sigma(i)}.
\end{equation*}
To obtain a formula for the false Sarrus rule we note that the product that corresponds to the diagonal from the upper left to the lower right beginning with $a_{1,k}$ is of the form
\[
\prod_{i=1}^na_{i, \mbox{\rm{mod}}(i+k-1,n)}\,=\,\prod_{i=1}^na_{i,\rho_k(i)}.\]
The respective permutation $\rho_k$ of $(1,\cdots,n)$ has the values
\[\rho_k(1)=k,\cdots,\rho_k(n-k+1)=n,\rho_k(n-k+2)=1,\cdots,\rho_k(n)=k-1.\]
It can be visualized as the rotation of a regular $n$-gon by an angle $(k-1)2\pi/n$ , where the vertices are numbered by $1,\cdots,n$.
\begin{center}
  \includegraphics{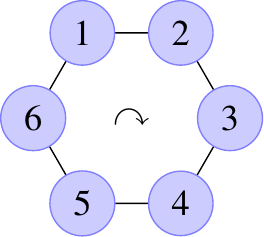}
  % \begin{tikzpicture}
  %   \draw (180:1)
  %   \foreach \i in {1,2,3,4,5,6} {-- ({60*(3-\i)}:1) node[circle,draw=blue!50,fill=blue!20]{\i} };
  %   \node at (0,0) {$\curvearrowright$};
  % \end{tikzpicture}
  \quad
  \raisebox{1cm}{$\stackrel{\rho_{3}}{\longrightarrow}$}
  \quad
  \includegraphics{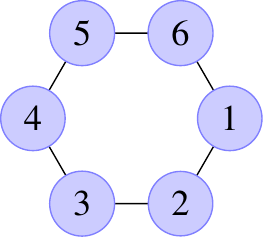}
  % \begin{tikzpicture}
  %   \draw (60:1)
  %   \foreach \i in {1,2,3,4,5,6} {--  ({60*(1-\i)}:1) node[circle,draw=blue!50,fill=blue!20]{\i}};
  % \end{tikzpicture}
\end{center}
These permutations will be called \emph{rotations} in the following.

The product that corresponds to the diagonal from the upper right to the lower left beginning with $a_{1,k}$ is of the form
\[
\prod_{i=1}^na_{i, \mbox{\rm{mod}}(k+1-i,n)}\,=\,\prod_{i=1}^na_{i,\mu_k(i)}.
\]
The respective permutation $\mu_k$ of $(1,\cdots,n)$ has the values
\[
\mu_k(1)=k,\,\mu_k(2)=k-1,\cdots,\mu_k(k)=1,\mu_k(k+1)=n,\cdots,\mu_k(n)=k+1.
\]
It can be visualized as the reflection of the above regular $n$-gon reflecting it in its axis of symmetry that lies in the middle of the vertices with numbers $1$ and $k$.
\begin{center}
  \includegraphics{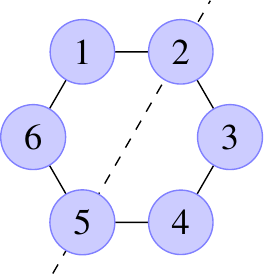}
  % \begin{tikzpicture}
  %   \draw[dashed] (240:1.6cm) -- (60:1.6cm);
  %   \draw (180:1)
  %   \foreach \i in {1,2,3,4,5,6} {-- ({60*(3-\i)}:1cm) node[circle,draw=blue!50,fill=blue!20]{\i}};
  % \end{tikzpicture}
  \quad
  \raisebox{1.3cm}{$\stackrel{\mu_{3}}{\longrightarrow}$}
  \quad
  \includegraphics{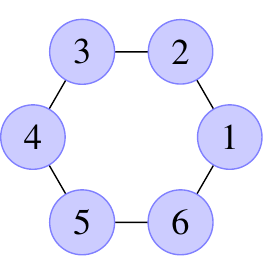}
  % \begin{tikzpicture}
  %   \draw[white,dashed] (240:1.6cm) -- (60:1.6cm);
  %   \draw(300:1)
  %   \foreach \i in {1,2,3,4,5,6} {-- ({60*(\i-1)}:1cm) node[circle,draw=blue!50,fill=blue!20]{\i}};
  % \end{tikzpicture}
\end{center}
These permutations will be called \emph{reflections} in the following.
\begin{exercise}
  Visualize all rotations $\rho_{k}$ and reflections $\mu_{k}$ for for $n=3,4,5$.
\end{exercise}

All involved permutations together form the \emph{dihedral group} $D_{n}$ (sometimes also denoted by $D_{2n}$ because it has $2n$ elements).
We assume that the reader is familiar with these groups.
In the following we mainly use the fact that the composition of two rotations and the composition of two reflections is a rotation, whereas the composition of a reflection and a rotation is a reflection. 
With the help of rotations $\rho_{k}$ and reflections $\mu_{k}$ we write the announced  formula for the False Sarrus Rule as
\begin{equation*}\label{f1}
  %\tag{False Sarrus}
  \dih(A)\,=\,\sum_{k=1}^{n}\prod_{i=1}^na_{i,\rho_{k}(i)}\,-\,\sum_{k=1}^{n}\prod_{i=1}^na_{i,\mu_{k}(i)}.
\end{equation*}
Since this quantity is built upon the dihedral group, we call it \emph{dihedrant}.
To push the similarity between the determinant and the dihedrant a little further, we define
\[
\sig(\rho_{k}) = 1,\qquad \sig(\mu_{k}) = -1,\quad k=1,\dots, n,
\]
and get the formula
\[
\dih(A) = \sum_{\sigma\in D_{n}}\sig(\sigma)\prod_{i=1}^{n}a_{i,\sigma(i)}.
\]
The above mentioned composition rules imply that for all $\sigma,\tau\in D_{n}$ it holds that
\[
\sig(\tau\circ\sigma) = \sig(\tau)\sig(\sigma).
\]
Let us consider the first cases:
\begin{itemize}
\item The group $D_1$ (visualized as a loop) consists of one
  reflection and the identity map, both permuting the vertices of the
  $1$-gon in the same way and thus, $\dih(A)=0$ for any $1\times 1$
  matrix $A$. 
\item For $n=2$, there are exactly two
  rotations and two reflections in $D_{2}$, where for both rotations there is a
  reflection that permutes the vertices in the same way, and hence
  $\dih(A) = a_{11}a_{22} + a_{12}a_{21} -
  a_{12}a_{21}-a_{11}a_{22}=\,0 $
  for any $2\times 2$ matrix $A$, too.
\item In the case $n=3$ we have that $D_3$ is the full symmetric group
  and the rotations fulfill $\sig(\rho_{k}) = \sgn(\rho_{k}) = +1$ and the reflections fulfill $\sig(\mu_{k}) = \sgn(\mu_{k}) = -1$
  and consequently,  the dihedrant equals the determinant.
\end{itemize}
The case $n=3$ is a lucky coincidence and for $n>3$ the dihedrant is
only loosely related to the determinant.  Nonetheless, the dihedrant
has some interesting properties.
Our first example for such a property says that the dihedrant also fulfills an equation similar to $\det(A) = \det(A^{T})$.

\begin{theorem}\label{thm:AT}
  $\dih(A^T)\,=\,\dih(A)$.
\end{theorem}
\begin{proof}
  Let $A^T = (b_{i,j})$. Since $A^T$ is the transposed of $A$, we have $b_{i,j}=a_{j,i}$. Hence,
  \begin{align*}
  \dih(A^T) & = \sum_{\sigma\in D_{n}}\sig(\sigma)\prod_{i=1}^{n}b_{i,\sigma(i)} = \sum_{\sigma\in D_{n}}\sig(\sigma)\prod_{i=1}^{n}a_{\sigma(i),i}\\
    & = \sum_{\sigma\in D_{n}}\sig(\sigma)\prod_{i=1}^{n}a_{i,\sigma^{-1}(i)}.
  \end{align*}
  Since $D_n$ is a group and the composition rule for $\sig$ implies $\sig(\sigma^{-1}) = \sig(\sigma)$, the last expression
  equals $\dih(A)$.
\end{proof}

For any matrix $A = (\mathbf{a}_{1},\dots,\mathbf{a}_{n})$ and any permutation $\sigma\in S_{n}$ it holds that $\det(\mathbf{a}_{\sigma(1)},\dots,\mathbf{a}_{\sigma(n)}) = \sgn(\sigma)\det(A)$.
However, the following example shows that this is not true for the dihedrant and arbitrary $\sigma\in S_{n}$, $n\geq 4$:
Consider the identity matrix $I$ and exchange the first two columns to get the matrix $\tilde I$.
Then we have $\det(\tilde I) = -1$, but $\dih(I)=1$ and $\dih(\tilde I) = 0$.
If we restrict ourselves to $\sigma\in D_{n}$, however, we still get a similar property.
\begin{theorem}\label{thm:rotation}
  For an $n\times n$ matrix $A= (\mathbf{a}_{1},\dots,\mathbf{a}_{n})$
  and $\sigma\in D_{n}$ it holds that
  \[
  \dih(\mathbf{a}_{\sigma(1)},\dots,\mathbf{a}_{\sigma(n)}) = \sig(\sigma)\dih(A).
  \]
\end{theorem}
\begin{proof}
  The proof follows immediately from
  \begin{align*}
    \dih(\mathbf{a}_{\sigma(1)},\dots,\mathbf{a}_{\sigma(n)}) & = \sum_{\tau\in D_{n}}\sig(\tau)\prod_{i=1}^{n}a_{i,(\tau\circ\sigma)(i)}\\
    & = \sig(\sigma)\sum_{\tau\in D_{n}}\sig(\sigma\circ\tau)\prod_{i=1}^{n}a_{i,(\tau\circ\sigma)(i)} = \sig(\sigma)\dih(A).
  \end{align*}
\end{proof}
Combining Theorem~\ref{thm:AT} and Theorem~\ref{thm:rotation}, we see that  a analogous statement is valid, if one replaces columns by rows.

% If we use formula (\ref{f1}) and Theorem 1 and consider $\dih$ as a map from $\left(\mathds{R}^n\right)^n$ to $\mathds{R}$, then we see that
% \begin{theorem}\label{thm:linear}
%   The map $\dih$ is a linear function of any row and any column of the matrix
%   $A$.
% \end{theorem}

Now, we consider $\det$ and $\dih$ as maps from $\left(\RR^n\right)^n$ to $\RR$, it becomes clear that both are multilinear:
\begin{theorem}\label{thm:linear}
  The map $\dih$ is linear in any row and in any column of the
  matrix  $A$.
\end{theorem}
\begin{proof} Let $j\in \{1,\cdots,n\}$ be fixed
  ${\bf{a}}_i^T=(a_{i,l}), i,l = 1,\cdots,n$, and
  ${\bf{b}}_j^T=(b_{j,l}),l = 1,\cdots,n$. The assertion concerning
  the rows follows from the following equation
  \begin{multline*}
    \dih((\mathbf{a}_{1},\dots,\alpha\mathbf{a}_{j} + \beta\mathbf{b}_{j},\dots,\mathbf{a}_{n})^{T})\! =\! \sum_{\sigma\in D_{n}}\sig(\sigma)\prod_{\substack{i=1\\ i\neq j}}^{n}a_{i,\sigma(i)}(\alpha a_{j,\sigma(j)} + \beta b_{j,\sigma(j)})\\
     = \alpha\sum_{\sigma\in D_{n}}\sig(\sigma)\prod_{i=1}^{n}a_{i,\sigma(i)} + 
    \beta\sum_{\sigma\in D_{n}}\sig(\sigma)\prod_{i=1,\ i\neq j}^{n}a_{i,\sigma(i)} b_{j,\sigma(j)}
  \end{multline*}
  Now, the assertion concerning the columns follows from Theorem~\ref{thm:AT}.
\end{proof}

There is an axiomatic characterization of the determinant as \emph{the} map from square matrices to numbers that is linear in every row, is invariant when adding one column to another, and gives the value one for the identity matrix (cf. the historic account~\cite{knobloch1994determinanttheory}). 
Note that the dihedrant only fulfills the first and the third of these determining properties, but not the second.

\section{When the False Sarrus Rule is actually right}

Of course, there are matrices $A$ for which $\dih(A) = \det(A)$; two obvious examples are $A=0$ and the identity matrix $I$ (for which $\dih(I) = \det(I)=1$), but there are more cases.

\subsection{Upper right and lower left triangular matrices}

Since for upper right and lower left triangular matrices both the determinant and the dihedrant are the product of the diagonal entries, we have $\dih(A) = \det(A)$ in this case. 
We can use Theorem~\ref{thm:rotation} and its analogon for rows of the matrix to build more examples of matrices where $\dih(A) = \det(A)$ and which are not easily recognized as such. One example is
\[
A =
\begin{pmatrix}
   1 &  0 &  0 & -1\\
   1 & -3 &  0 & -3\\
   1 &  1 &  5 &  5\\
   0 &  0 &  0 &  1
\end{pmatrix}
\]
with $\det(A) = \dih(A) = -15$ (Hint: Permute the last column to the first and the last row to the first.). However, there are also matrices $A$ for which $\dih(A) = \det(A)\neq 0$ which are not obtained by submitting indices of rows and columns of triangular matrices to rotations and reflections. One example we found by trial and error is
\[A =
\begin{pmatrix}
  2 & 2 & 2 & 2\\
  1 & 2 & 1 & 1\\
  2 & 2 & 2 & 1\\
  1 & 2 & 2 & 1
\end{pmatrix}
\]
with $\dih(A) = \det(A) = 2$.

\subsection{Rank deficient matrices}

Here are some classes of matrices with $\dih(A) = \det(A) = 0$:
% It is obvious from (\ref{f1}) that $\dih(A)=0$ if all rows are identical. Now, we use Theorem 4 to see that
% \begin{theorem}
%   If the rows of $A$ are linear dependent on one row, then $\dih(A)=0$.
% \end{theorem}

% An equivalent formulation of this theorem is
\begin{theorem}\label{thm:rank1}
  If $\rank(A) =1,$ then $\dih(A)=0$.
\end{theorem}
\begin{proof}
  Let $A\,=\ ({\bf{a}}_1,\cdots,{\bf{a}}_n),$ and
  $ {\bf{a}}_i=\alpha_i\, {\bf{v}}, {\bf{v}} \in \mathds{R}^n,
  i=1,\cdots,n$. Then, according to Theorem~\ref{thm:linear},
  \[
  \dih(A)\,=\,\left(\prod_{i=1}^n\alpha_i\right)\dih({\bf{v}},\cdots,{\bf{v}}).\]
  We have $\dih({\bf{v}},\cdots,{\bf{v}})$$\,=\,
  0$, as all products in the False Sarrus Rule have identical values.
\end{proof}

% It may be surprising that even more is true.
% \begin{theorem}\label{thm:rank2}
%   If $n=4$ or $n=5$, $\rank(A) =2,$ then $\dih(A)=0$.
% \end{theorem}

% \begin{proof}
%   Again using Theorem~\ref{thm:linear} on linearity in all rows and
%   Theorem~\ref{thm:AT}, it is obvious that it suffices to prove that
%   $\dih(A)=0$, if $A$ has only two sorts of rows which are linear
%   independent. In the cases $n=4$ and $n=5$, we may assume that two
%   rows are identical and the rest is of another shape.
  
%   Now, we find to any $\rho \in R_n$ in formula (\ref{f1}) a
%   $\mu \in M_n$ such that
%   $\prod_{i=1}^na_{i,\rho(i)}\,=\,\prod_{i=1}^na_{i,\mu(i)}.$ This can
%   be done as follows. Let the two identical vectors are in the rows
%   with numbers $i$ and $j=i+k, k\geq 1$. Now we consider the
%   coefficient $a_{i,\rho(i)}$. Since $\rho$ is a rotation, we have that $\rho(j) = \rho(i+k) = \bmod(\rho(i)+k,n)$. Hence, in the product belonging to $\rho$ we find the coefficient $a_{j, \rho(i)+k}$ where
%   $ \rho(i)+k$ has to be considered modulo $n$. The apt reflection
%   $\mu$ has to be chosen such that $\mu(j)=\rho(i)$ and this is always possible for $n=4$ and $n=5$. Hence,
%   $\dih(A)=0$.
% \end{proof}

Now, we turn to the consideration of matrices that have the rank
two. In this case all rows of the matrix  are of the form
$\alpha_{i} \mathbf{a}^{T}
+ \beta_{i}
\mathbf{b}^T$, $i=1,\dots,n$, and, by Theorem~\ref{thm:linear} it holds that
\[
\dih(A) 
  = \alpha_{1}\dih
  \begin{pmatrix}
    \mathbf{a}^{T}\\
    \alpha_{2}\mathbf{a}^{T} + \beta_{2}\mathbf{b}^{T}\\
    \vdots\\
    \alpha_{n}\mathbf{a}^{T} + \beta_{n}\mathbf{b}^{T}x
  \end{pmatrix}
  +\beta_{1} \dih
  \begin{pmatrix}
    \mathbf{b}^{T}\\
    \alpha_{2}\mathbf{a}^{T} + \beta_{2}\mathbf{b}^{T}\\
    \vdots\\
    \alpha_{n}\mathbf{a}^{T} + \beta_{n}\mathbf{b}^{T}
  \end{pmatrix}.
\]
Applying the same rule over and over again, we end up with a linear combination
of dihedrants of matrices where the rows of each matrix are either $\mathbf{a}^{T}$ or $\mathbf{b}^{T}$.
\begin{exercise}
  Assume that in the above matrix $A$ we have $\mathbf{a}^{T}\neq 0$, $\mathbf{b}^{T}\neq 0$, and $\mathbf{a}^{T}\neq \mathbf{b}^{T}$. How many summands will the above mentioned linear combination have?
\end{exercise}
Our next step is to analyze the dihedrant of a matrix for which $k$ rows are $\mathbf{a}^{T}$ and $n-k$ rows are $\mathbf{b}^{T}$.
If
$k=0$ or $k=n,$ then those dihedrants are zero according to Theorem~\ref{thm:rank1}. For $k=1$ or $k=n-1$ we have the following result:

\begin{theorem}\label{thm:n-1-row-equal}
  If $n-1$ rows of the matrix $A$ are identical, then $\dih(A)=0$.
\end{theorem}
\begin{proof}
  Let $n-1$ rows be equal to ${\mathbf{b}}^T=(b_1,\cdots,b_n)$ and let
  the remaining row be ${\mathbf{a}}^T=(a_1,\cdots,a_n)$. Then every
  product in the False Sarrus Rule is of the form
  $a_i\prod_{j=1,j\neq i}^nb_j$. Any of these products appears twice,
  once corresponding to a rotation, and once to a reflection. Hence, they cancel which implies
  $\dih(A)=0$.
\end{proof}

For $k=2$ or $k=n-2$ the same is still true:
\begin{theorem}\label{n-2-row-equal}
  If $n-2$  rows of the matrix $A$ are identical, and the
  remainig two rows are identical, too, then $\dih(A)=0$.
\end{theorem}
\begin{proof}
  Let the $n-2$ identical rows be equal to
  ${\mathbf{b}}^T=(b_1,\cdots,b_n)$ and let the remaining rows be
  equal to ${\mathbf{a}}^T=(a_1,\cdots,a_n)$. Then every product in the
  False Sarrus Rule is of the form
  $a_ia_j\prod_{l=1,l\neq i,j}^nb_l$. If a product of this form
  appears in the False Sarrus Rule, it appears twice. Similar to Theorem~\ref{thm:n-1-row-equal}, one product corresponds to a rotation, the other to a reflection; again they cancel and we obtain $\dih(A)=0$.
\end{proof}

If only $n-3$ rows are identical and the remaining three or more rows
are identical, this reasoning is no longer valid. This is exemplified
by the following $6\times 6$ example 
\[
A\,=\,\left(\begin{array}{cccccc} 1 & 1 & 0 & 0 & 1 & 0\\1 & 1 & 0 & 0 & 1 & 0\\1 & 1 & 1 & 1 & 1 & 1\\ 1 & 1 & 1 & 1 & 1 & 1\\1 & 1 & 0 & 0 & 1 & 0\\1 & 1 & 1 & 1 & 1 & 1\end{array}\right).
\]
Here $\dih(A)=\,1,$ and $\rank(A)=2$.

Combining Theorems~\ref{thm:n-1-row-equal} and~\ref{n-2-row-equal} for $n=4$ and $n=5$ we obtain:

\begin{corollary}
  If $n=4$ or $n=5$, and $\rank(A)\leq 2$, then
  $\dih(A)=0$.
\end{corollary}
As the following matrix shows, $4\times 4$ matrices $A$ with $\rank(A) =3$ do not necessarily have dihedrant equal to zero
\[
A\,=\,\left(\begin{array}{cccc} 1 & 2 & 3 & 4\\1 & 2 & 3 & 4\\1 & 0 &
    0 & 0\\0 & 0 & 0 & 1\end{array}\right).
\]
Indeed, it holds that $\dih(A)=-6$, but $\det(A)=0.$

\subsection{Upper left and lower right triangular matrices}

To see when upper left or lower right triangular matrices have equal
determinant and dihedrant, it is helpful to calculate the signs of
rotations and reflections. 

\begin{lemma}
  \label{lem:signs-refl-inv}
  Let $\rho_k$  and $\mu_k, k=1,\cdots,n$, be as above, and
  \[
  N(\rho_k)\,:=\,(k-1)(n-k+1),\]
  and
  \[
  N(\mu_k)\,:=\,\frac{(n-k-1)(n-k)}{2}\,+\,\frac{k(k-1)}{2}.\]
  Then $\sgn(\rho_k)\,=\,(-1)^{N(\rho_k)}$ and $\sgn(\mu_k)\,=\,(-1)^{N(\mu_k)}$.
  % If $\rho_{k} \in D_n$ is the rotation by an angle
  % $\frac{2k\pi}{n}, k=0,\dots, n-1,$ then the number $N(\rho_{k})$ of
  % inversions is given by
  % \[
  % N(\rho_{k})\,=\,k(n-k).
  % \]
  % Hence $\sgn(\rho_{k})\,=\,(-1)^{k(n-k)}$.
  
  % If $\mu_{k} \in D_n$ is the reflection that interchanges the vertices
  % with numbers $n$ and $k, k=1,\dots, n,$ then the number $N(\mu_{k})$
  % of inversions is given by
  % \[
  % N(\mu_{k})\,=\,\frac{(n-k)(n-k+1)}{2}\,+\,\frac{(k-1)(k-2)}{2}.
  % \]
  % Hence,
  % $\sgn(\mu_{k})\,=\,(-1)^{N(\mu_{k})}$.
\end{lemma}
\begin{proof}
  % Since for the rotation in question we have
  % \[ \rho_{k}(1)=k+1,\dots, \rho_k(n-k)=n, \rho_k(n-k+1)=1,\dots,
  % \rho_k(n)=k,\]
  % it is obvious that the number of inversion is $N(\rho_{k}) = k(n-k)$.
  
  % Since for $k=1,\dots,n-1$ the reflection $\mu_{k}$ interchanges the vertices with numbers
  % $n$ and $k$, we have
  % \[ \mu_{k}(1)=k-1, \mu_k(2)=k-2,\dots, \mu_k(k-1)=1, \mu_k(k)=n,\dots,
  % \mu_k(n)=k.\]
  % Using the famous
  % formula
  % $\sum_{i=1}^ji\,=\,\tfrac{j(j+1)}{2}$ it is easy to count the number of inversions and
  % this results in the formula for $N(\mu_{k})$ in these cases.  For the remaining
  % reflection, we get
  % \[
  % \mu_n(i)=n-i,i=1,\dots, n-1, \mu_n(n)=n.\]
  % Here, it is also easy to count the number of inversions to get $N(\mu_{n}) = \tfrac{(n-2)(n-1)}2$.
  If we consider the values $(k,k+1,\cdots,n,1,\cdots,k-1)$ of
  $\rho_k$, we see that $n-k+1$ transpositions are needed to permute $1$
  to the first place. For the numbers $2,\cdots,k-1$ we also need
  $n-k+1$ to gain the natural order. This implies the formula for
  $N(\rho_k)$.
  
  If we consider the values $(k,k-1,\cdots,1,n,\cdots,k+1)$ of
  $\mu_k$, we see that $k-1$ transpositions are needed to place $1$
  onto the first place, $k-2$ transpositions for the number $2$, and so
  on. Hence, we need
  \[\sum_{j=1}^{k-1}j =\frac{k(k-1)}{2}\]
  transpositions for the first $k$ values and likewise
  $\frac{(n-k-1)(n-k)}{2}$ transpositions for the last $n-k-1$
  elements. This proves the formula for $N(\mu_k)$.
\end{proof}

\begin{exercise}
  Calculate a more explicit formula for the $\sgn$ of rotations and reflection, i.e. $\sgn(\rho_{k})$ and $\sgn(\mu_{k})$, or, put differently, determine for which $n$ and $k$ it holds $\sgn(\rho_{k}) = \sig(\rho_{k})$, $\sgn(\mu_{k}) = \sig(\mu_{k})$. Hint: Use the remainder of $n$ and $k$ by division by two for the rotations and the remainder of $n$ and $k$ by division by four for the reflections.
\end{exercise}
In the case of an upper left or lower right triangular matrix with non-zero entries on the anti-diagonal, the False Sarrus Rule only has one non-zero summand (the product of the entries of the anti-diagonal) and we get
\[
\dih(A)\,=\,-\prod_{i=1}^na_{i,(n-i+1)}.
\]
For the determinant of such a matrix there is also only one non-zero
summand in Leibniz' formula (the same product as for the dihedrant)
which is the one that corresponds to the reflection $\mu_{n}$. By the
second part of Lemma~\ref{lem:signs-refl-inv} the sign of this
summand is
\[
\sgn(\mu_{n}) =
\begin{cases}
  \hphantom{-}1, & \text{for $\tfrac{n(n-1)}2$ even},\\
  -1, & \text{for $\tfrac{n(n-1)}2$ odd},\\
\end{cases}
\]
and thus,
\[ 
\det(A)=
\begin{cases}
  -\prod_{i=1}^na_{i,(n-i+1)}, & \text{if $\bmod(n,4) = 2$ or $\bmod(n,4)=3$}\\
  \hphantom{-}\prod_{i=1}^na_{i,(n-i+1)}, & \text{if $\bmod(n,4)= 0$ or $\bmod(n,4)=1$}.
\end{cases}
\]
We conclude
\begin{theorem}
  Let $A$ be an upper left or lower right triangular $n\times n$ matrix. Then it holds that
  $\dih(A) = \det(A)$ if $\bmod(n,4)= 2$ or $\bmod(n,4)=3$.
\end{theorem}

\begin{exercise}
  Work out under which conditions there holds $\dih(A) = \det(A)$ where $A$ is an $n\times n$ matrix with the following pattern of non-zero entries:
\begin{equation*}
  \begin{pmatrix}
    * & *      &        & 0\\
    & \ddots & \ddots & \\
    & 0      & \ddots & *\\
    * &        &        & *  
  \end{pmatrix}
\end{equation*}
\end{exercise}
% \subsection{Matrices with two ``full'' diagonals}
% Now, we consider matrices $A$ wherein $a_{i,i}\neq 0,i=1,\dots, n, a_{i,i+1}\neq 0,i=1,\dots n-1, a_{n,1}\neq 0$, and $a_{i,k}=0$ for $i,k$ otherwise, i.e., the matrices we consider have the form
% \begin{equation}
%   \begin{pmatrix}
%     * & *      &        & 0\\
%     & \ddots & \ddots & \\
%     & 0      & \ddots & *\\
%     * &        &        & *  
%   \end{pmatrix}\label{eq:two-full-diagonals}
% \end{equation}
% In this case one readily calculates the dihedrant as
% \[\dih(A)\,=\,\prod_{i=1}^na_{i,i}\,+\,a_{n,1}\,\prod_{i=1}^{n-1}a_{i,i+1}.\]
% To compare this result with the determinant, we use the first part of Theorem 7. It indicates that the sign corresponding to the rotation by an angle $\frac{2\pi}{n}$ is $+1$ if and only if $n$ is odd. 
% Using Laplace's method of cofactor expansion to calculate the determinant we conclude
% \begin{theorem}
%   An $n\times n$ matrix $A$ of the form~(\ref{eq:two-full-diagonals}) fulfills $\dih(A) = \det(A)$ if $n$ is odd. If $n$ is even there exists an $A$ such that $\dih(A) \neq \det(A)$.
% \end{theorem}

% \begin{exercise}
%   Consider matrices $A$ wherein
%   $a_{(n-i+1),i}\neq 0, i=1,\dots, n, \,\,a_{(n-i+2),i}\neq 0,
%   i=2,\dots ,n,a_{1,1}\neq 0$,
%   and $a_{i,k}=0$ for $i,k$ otherwise. Use the second part of Theorem
%   7 and the Laplace rule to identify those cases where
%   $\dih(A)\,=\,\det(A)$.
% \end{exercise}

\section{A modified Sarrus' rule for $4\times4$ matrices.}

The most obvious hint that the dihedrant fails to be equal to the determinant in general is that it has way too few summands. This is due to the fact that the dihedral group $D_{n}$ only has $2n$ elements, while the symmetric group $S_{n}$ has $n!$.
For the case $n=4$ we need $4!=24$ summands and the False Sarrus Rule only provides $8$ of them
(and half of them even have the wrong sign).
Here is a scheme that provides all $24$ summands with the correct signs:
\begin{center}
  \includegraphics{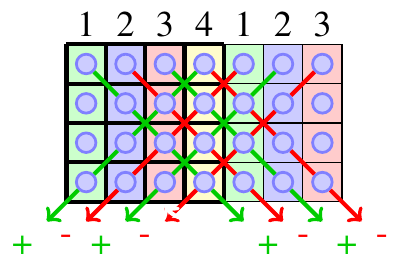}%
  \includegraphics{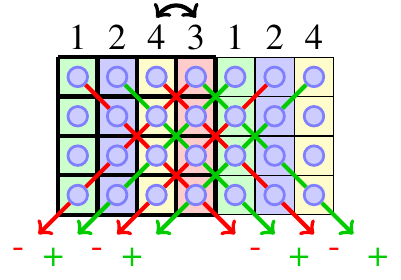}%
  \includegraphics{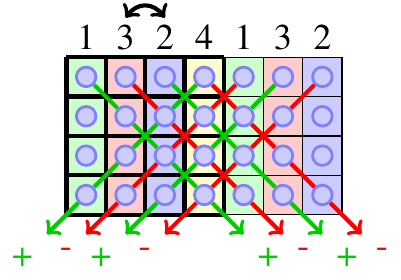}
\end{center}
We do not show a similar rule for the cases $n\geq 5$, since one needs $\tfrac{(n-1)!}{2}$ permutations of columns to get the needed total of $n!$ permutations.
We are by far not the first, to propose such an approach to memorize determinants.
The oldest reference we could track down is~\cite{arschon1935sarrus} (Russian original~\cite{arschon1935}) and a recent reference is~\cite{sobamowo2016extensionsarrusrule}.
\bigskip

%\bibliographystyle{plain}
%\bibliography{det}

\textbf{Acknowledgement.}  The authors thank the user Douglas Zare from \url{https://mathoverflow.net}
  who provided an answer to an initial question on this topic.

\end{document}